\documentclass{amsart}
\usepackage{mathrsfs}
\usepackage{amssymb,latexsym,amsmath,amsthm,enumitem,mathrsfs}
\usepackage{hyperref}
\usepackage{color}
\usepackage{stmaryrd}
\usepackage{comment}
\usepackage{mathrsfs}
\usepackage{lineno}
\usepackage{bbm}

\usepackage{pgfplots}
\usepackage{graphicx}
\usetikzlibrary{math}

\usepackage{scalerel,stackengine}
\usepackage{tikz}

\usepackage{ esint }
\usepackage{graphicx}
\usepackage{amssymb}
\newtheorem{theorem}{Theorem}
\newtheorem{lemma}[theorem]{Lemma}
\newtheorem{remark}[theorem]{Remark}

\newtheorem{definition}[theorem]{Definition}

\newtheorem{corollary}[theorem]{Corollary}

\theoremstyle{definition}

\newcommand{\Bor}{\mathrm{Bor}}

\newcommand{\cC}{\mathcal{C}}
\newcommand{\cD}{\mathcal{D}}

\newcommand{\spt}{\mathrm{spt}}

\newcommand{\N}{\mathbb N}
\newcommand{\T}{\mathbb T}

\renewcommand{\Re}{\text{Re}}

\newcommand{\norm}[1]{ \|  #1 \|}

\def\set4{\mathcal I}
\def\tup14{(1,2,3,4)}

\def\eps{\varepsilon}

\newcommand{\hatf}{\hat{f}}
\newcommand{\hatg}{\hat{g}}
\newcommand{\hath}{\hat{h}}
\newcommand{\cchi}{\Check{\chi}}
\newcommand{\cpsi}{\Check{\psi}}
\newcommand{\hchi}{\hat{\chi}}
\newcommand{\lesim}{\lesssim}
\newcommand{\RapDec}{\mathrm{RapDec}}
\newcommand{\bT}{\mathbf{T}}
\newcommand{\gesim}{\gtrsim}

\newtheorem*{comm*}{Comment}
\newtheorem*{lemma*}{Lemma}
\newtheorem{thm}{Theorem}[section]

\newtheorem{cor}[thm]{Corollary}

\newtheorem{prop}[theorem]{Proposition}

\newcommand{\R}{\mathbb{R}}

\newcommand{\de}{\delta} 

\begin{document}

 \author{Paige Bright}
 \author{Yuqiu Fu}
 \author{Kevin Ren}
 \address{Department of Mathematics,
 Massachusetts Institute of Technology}
 \email{paigeb@mit.edu}
 \email{yuqiufu@mit.edu}
 \address{Department of Mathematics, Princeton University}
 \email{kevinren@princeton.edu}


\date{}

\title{Radial Projections in $\R^n$ Revisited}

\begin{abstract}
We generalize the recent results on radial projections by Orponen, Shmerkin, Wang \cite{orponen2022kaufman} using two different methods. In particular, we show that given $X,Y\subset \R^n$ Borel sets and $X\neq \emptyset$. If $\dim Y \in  (k,k+1]$ for some $k\in \{1,\dots, n-1\}$, then 
\[
\sup_{x\in X} \dim \pi_x(Y\setminus \{x\}) \geq \min \{\dim X + \dim Y - k, k\}.
\]
Our results give a new approach to solving a conjecture of Lund-Pham-Thu in all dimensions and for all ranges of $\dim Y$. 

The first of our two methods for proving the above theorem is shorter, utilizing a result of the first author and Gan \cite{bright2023exceptional}. Our second method, though longer, follows the original methodology of Orponen--Shmerkin--Wang, and requires a higher dimensional incidence estimate and a dual Furstenberg-set estimate for lines. These new estimates may be of independent interest.
\end{abstract}

\maketitle

\section{Introduction}

In this paper, we study the relationship between Hausdorff dimension and \textit{radial projections} $\pi_x:\R^n \setminus x \to \mathbb{S}^{n-1}$, defined via the equation 
\[
\pi_x(y) := \frac{y-x}{|y-x|}, \hspace{1cm} y\in \R^n \setminus \{x\}.
\]

In particular, we generalize recent work of Orponen, Shmerkin, and Wang to higher dimensions. Here is our main result:
\begin{theorem}\label{main}
    Let $X,Y\subset \R^n$ be Borel sets with $X\neq \emptyset$ and $\dim Y \in (k,k+1]$ for some $k \in \{1,\dots, n-1\}$. Then,
    \[
    \sup_{x\in X} \dim \pi_x(Y\setminus \{x\}) \geq \min\{\dim X + \dim Y - k,k\}.
    \]
\end{theorem}

Note that the case when $k=1$ and $n=2$ was solved by Orponen, Shmerkin, and Wang in \cite{orponen2022kaufman}. Similarly, using an analogous method to their Corollary 1.6, we obtain the following result, which was first proved as \cite[Theorem 4.1]{orponen2022kaufman}.
\begin{corollary}\label{lundphamthu}
    Let $Y\subset \R^n$ be a Borel set with $\dim Y\in (k,k+1]$ for some $k \in \{1,\dots, n-1\}.$ Then,
    \[
    \dim \underbrace{(\{x\in \R^n : \dim \pi_x(Y\setminus \{x\}) <s \})}_{X} \leq \max\{k+s - \dim Y, 0\}
    \]
    for all $0\leq  s \leq k$.
\end{corollary}

\begin{proof}
    First, suppose $s<k$, and assume for the sake of contradiction that the exceptional set $X$ satisfies the following inequality:
    \[
    \dim X > \max\{k + s - \dim Y, 0\}.
    \]
    Then, note that $X\neq \emptyset$ as $\dim X>0$. Hence, applying Theorem \ref{main}, we have 
    \[
    \sup_{x\in X} \dim \pi_x(Y\setminus \{x\}) \geq \min\{\dim X + \dim Y - k,k\}> s.
    \]
    This is a contradiction with the definition of $X$. In the case $s=k$ simplifies down to noticing $X$ may be written as the following countable union $$X = \bigcup_{j=1}^\infty \{x\in \mathbb{R}^n : \dim \pi_x(Y\setminus \{x\}) < k - 1/j\}.$$
\end{proof}

We prove Theorem \ref{main} in two ways. The first method for proving our main theorem is a direct proof utilizing a result of the first author and Gan \cite[Theorem 1]{bright2023exceptional}, providing a shorter (generalized) proof of \cite[Theorem 1.2]{orponen2022kaufman}. Our second method, though longer, follows the original methodology of Orponen--Shmerkin--Wang's Theorem 1.2. Orponen, Shmerkin, and Wang first reduce the problem to a set of points and tubes with a ``Furstenberg'' type restriction (see \cite[Theorem 3.1]{orponen2022kaufman}). Then, they apply Fu and Ren's planar incidence estimate (see \cite[Theorem 1.5]{fu2022incidence}) and a dual Furtenberg set estimate (see \cite[Lemma 3.13]{orponen2022kaufman}) to complete the proof. The (second) proof of our main result in this paper follows in a similar manner, and thus the largest barriers to proving Theorem \ref{main} using Orponen, Shmerkin, and Wang's framework are 
\begin{itemize}
\item Generalizing the Fu-Ren incidence estimate between points and tubes to points and tubes in $\R^n$.
\item Proving the (generalized) dual Furstenberg estimate in $\R^n$.
\end{itemize}

We find these two estimates, which may be of independent interest, and obtain the following:

\begin{theorem}
    Let $\alpha\in [0,n]$ and $\beta\in [0,2(n-1)]$. Then, for every $\epsilon>0$, there exists a $\de_0 = \de_0(\epsilon)>0$ such that the following holds for all $\de \in (0,\de_0]$. If $P\subset B^n$ is a $(\de,\alpha,\de^{-\epsilon})$-set with $|P| \sim \delta^{-\alpha}$, and $\T$ is a $\de$-separated $(\de, \beta,\de^{-\epsilon})$-set with $|\T| \sim \delta^{-\alpha}$, then
    \[
    I(P,\T) \leq |P||\T| \de^{\kappa (\alpha+\beta-n+1) - 5\epsilon}
    \]
where $\kappa = \min\left\{\frac{n-1}{\alpha + \beta -n+1}, \frac{1}{2}\right \}$.
\end{theorem}

\begin{theorem}\label{dualFurstEst}
    Let $\alpha \in [0,n-1]$ and $\beta \in (0,n]$. Suppose $E$ is a non-empty $\beta$-dimensional subset of $\R^n$ and $\mathcal{L}$ is a collection of lines with the property that for every point $x \in E$ there exists a $\alpha$-dimensional collection of lines $\mathcal{L}_x \subset \mathcal{L}$ such that every line in $\mathcal{L}_x$ passes through $x$. Then $$\dim \mathcal{L} \geq \gamma(\alpha, \beta),$$ where $\gamma(\alpha, \beta) = \alpha + \min \{\alpha, \beta \}$. 
\end{theorem}

\begin{remark}
\rm{
    Note that Theorem \ref{dualFurstEst} is derived via a special case of \cite[Corollary 2.10]{ren2023discretized} and a Hausdorff content argument based on \cite{hera2020improved}.}
\end{remark}

\noindent These two new results are the focus of Sections \ref{sec3} and \ref{sec4} respectively. After generalizing these results, the bootstrapping scheme of Orponen--Shmerkin--Wang carries over directly to proving Theorem \ref{main} when $k=n-1$ in Section \ref{sec5}. We then deduce all other values of $k$ using Marstrand's projection theorem in Section \ref{sec6}.

\subsection{Recent Developments in Radial Projections} \label{ss:radial projections}

In recent years, many have been studying the effects of radial projections on Hausdorff dimension. The most recent thread of theorems started with the study of radial projections in finite fields studied by Lund, Pham, and Thu \cite{lund2022radial}. In particular, they showed that given $Y\subset \mathbb{F}_q^n$ and $M$ a positive integer with $|Y| \gtrsim q^{n-1}$ and $M \lesssim q^{n-1}$, then 
\[
|\{x\in \mathbb{F}_q^n : |\pi_x(Y)|\leq M\}| \lesssim q^{n-1} M |Y|^{-1}.
\]
This led to the following conjecture in $\R^n$ (when $k=n-1$) put forward by Lund--Pham--Thu, which was later resolved by the first author and Gan:
\begin{theorem}[\cite{bright2023exceptional}, Theorem 1]\label{RnLundPhamThu}
    Let $Y\subset \R^n$ be a Borel set with $\dim Y\in (k,k+1]$ for some $k\in \{1,\dots, n-1\}$. Then, for $0<s<k$
    \[
    \dim (\{x\in \R^n \setminus Y: \dim \pi_x(Y)<s\}) \leq \max\{k + s - \dim Y,0\}
    \]    
\end{theorem}

One can prove the following result (conjectured by Liu \cite{liu2019hausdorff}) as a consequence of Theorem \ref{RnLundPhamThu}. 

\begin{theorem}[\cite{bright2023exceptional}, Theorem 2]\label{liu}
    Let $Y\subset \R^n$ be a Borel set with $\dim Y \in (k-1,k]$ for some $k \in \{1,\dots, n-1\}$. Then we have 
    \[
    \dim (\{x\in \R^n \setminus Y : \dim \pi_x(Y)< \dim Y\}) \leq k.
    \]
\end{theorem}

Theorem \ref{liu} followed via a dimensional reduction by the first author and Gan, as well as a ``swapping trick'' of Liu. Note that finite field analogues of the above two theorems has recently been proven in \cite{bright2023radial} by the first author, Lund, and Pham. Furthermore, shortly after the first author's and Gan's proofs, Orponen, Shmerkin, and Wang proved a plethora of new results regarding radial projections (including Theorems \ref{RnLundPhamThu} and \ref{liu}). Among these new results were the following two theorems:

\begin{theorem}[\cite{orponen2022kaufman}, Theorem 1.1]\label{OSW1.1}
    Let $X\subset \R^2$ be a (non-empty) Borel set which is not contained in any line. Then, for every Borel set $Y\subset \R^2$,  
    \[
    \sup_{x\in X} \dim \pi_x(Y\setminus \{x\}) \geq \min \{\dim X, \dim Y,1\}.
    \]
\end{theorem}

\begin{theorem}[\cite{orponen2022kaufman}, Theorem 1.2]\label{OSW1.2}
Let $X,Y\subset \R^2$ be Borel sets with $X\neq \emptyset$ and $\dim Y >1$. Then,  
    \[
    \sup_{x\in X} \dim \pi_x(Y\setminus \{x\}) \geq \min \{\dim X+ \dim Y-1,1\}.
    \]
\end{theorem}

Theorem \ref{OSW1.1} generalized to higher dimensions for Borel subsets $Y\subset \R^n$ with $\dim Y>k-1/k- \eta$ for $\eta>0$ sufficiently small. This has been strengthened to all $Y$ with $\dim Y>k-1$ (by the third author in \cite{ren2023discretized}), as was conjectured by Orponen, Shmerkin, and Wang. On the other hand, Theorem \ref{OSW1.2} was not generalized to higher dimensions. This generalization to $\R^n$ is the focus of this paper.

\subsection{Notation and Preliminaries}

The notation $B(x,r)$ denotes a closed ball of radius $r>0$ with center $x\in X$ in a metric space $(X,d)$. If $X = \R^n$, we let $B^n = B(0,1)$. If $A\subset X$ is a bounded set, and $r>0$, we let $|A|_r$ denote the $r$-covering number of $A$, i.e. the minimum number of closed balls of radius $r$ required to cover $A$. We denote the cardinality of a set $|A|$.

Given $X,Y$ are positive numbers, then $X\lesssim Y$ means that $X\leq C Y$ for some constant $C$, while $X\gtrsim Y$ and $X\sim Y$ stand for $Y\lesssim X$ and $X\lesssim Y \lesssim X$ respectively.

If $\mu$ is a positive finite measure on $\R^n$ and $t\geq 0$, the $t$-energy of the measure $\mu$ is defined via the equation 
\[
I_t(\mu) = \iint |x-y|^{-t} d\mu(x) d\mu(y) \in (0,\infty].
\]
Frostman's Lemma states that if $X\subset \R^n$ is a Borel set and $\dim X >t\geq 0$, then there exists a Borel probability measure $\mu$ with $\spt(\mu) \subset X$ and $I_t(\mu)<\infty$. See \cite[Chapter 8]{mattila_1995} for more details.

Lastly, we define $(\de,s)$-sets and Katz-Tao $(\de,s)$-sets.

\begin{definition}[$(\de, s, C)$-sets] Let $(X,d)$ be a metric space, let $P\subset X$ be a set, and let $\de, C>0$, $s\geq 0$. Then, we say that $P$ is a $\mathbf{(\de, s, C)}$\textbf{-set} if 
\[
|P\cap B(x,r)|_\de \leq C r^s |P|_\de
\]
for all $x\in X$ and all $r\geq \de$.    
\end{definition}

\begin{definition}[Katz-Tao $(\de, s,C)$-sets] Let $(X,d)$ be a metric space. We say that a $\de$-separated set $P\subset X$ is a \textbf{Katz-Tao} $\mathbf{(\de,s,C)}$-\textbf{set} if 
\[
|P \cap B(x,r)| \leq C\left(\frac{r}{\de}\right)^{s}
\]
for all $x\in X$ and all $r\geq \de$.
\end{definition}

\bigskip
\begin{sloppypar}
\noindent {\bf Acknowledgements.} A number of ideas for this paper were conceived while the first author was working on the Study Guide Writing Workshop 2023 at UPenn. She would like to thank her collaborators on the study guide: Ryan Bushling, Caleb Marshall, Alex Ortiz, and their mentor Josh Zahl for insightful discussions. The third author is supported by a NSF GRFP fellowship and a grant from the AIM research community on Fourier restriction.
\end{sloppypar}

\section{A Direct Proof of Theorem \ref{main} when $X\cap Y = \emptyset$}\label{sec2}

In this direct proof, we use Theorem \ref{RnLundPhamThu}.

\begin{proof}[Proof of Theorem \ref{main}, $X\cap Y = \emptyset$]
Suppose that $X,Y\subset \R^n$ are Borel sets with  $\dim Y\in (k,k+1]$, $X \neq \emptyset$, and $X\cap Y = \emptyset$. If $\dim X = 0$, then we wish to prove that
\begin{equation*}
    \sup_{x \in X} \dim (\pi_x (Y \setminus \{ x \})) \ge \dim Y - k.
\end{equation*}
In fact, we will show that $\dim (\pi_x (Y \setminus \{ x \})) \ge \dim Y - 1$ for any $x \in X$. By Frostman's lemma, for any $s < \dim Y$, there exists a $s$-dimensional Frostman probability measure $\mu$ supported on $Y$. Now for any ball $I$ of diameter $r$ on the unit sphere $S^{n-1}$, we have that $T := \pi_x^{-1} (I) \cap B(0, 1)$ is contained in $r^{-1}$ many $r$-balls, so $\mu(T) \lesim r^{s-1}$ by the Frostman condition. Hence, the pushforward measure $\pi_x \mu$ satisfies a $(s-1)$-dimensional Frostman condition, so by Frostman's lemma again, we have that $\dim (\pi_x (Y)) \ge s-1$. Since $s < \dim Y$ was arbitrary, we have $\dim (\pi_x (Y)) \ge \dim Y - 1$, as desired.

Thus, assume $\dim X > 0$. Let $$s = \min \{\dim X + \dim Y - k,k\}.$$ Let $\epsilon>0$ be small enough such that $s - \epsilon > 0$ and $\dim X - \epsilon>0$, which we may assume as $\dim Y>k$ and $\dim X>0$. Under these assumptions on $\epsilon$, we may apply Theorem \ref{RnLundPhamThu} to $s-\epsilon$, we obtain 
    \[
        \dim (\{x \in \R^n \setminus Y : \dim \pi_x(Y) < s-\epsilon\}) \leq \max\{k+ s-\epsilon - \dim Y, 0\} < \dim X.
    \]
    Given $\dim (\{x \in \R^n \setminus Y : \dim \pi_x(Y) < s-\epsilon\})$ is strictly less than $\dim X$, it follows that $X\not\subseteq \{x\in \R^n \setminus Y : \dim \pi_x(Y) < s-\epsilon\}$. Given $X\cap Y = \emptyset$, this implies that for every $\epsilon>0$ sufficiently small, there exists an $x\in X$ such that 
    \[
    \dim \pi_x(Y\setminus \{x\}) \geq s-\epsilon \    = \min \{\dim X + \dim Y - k, k\} - \epsilon.
   \]
    Given $\epsilon$ was arbitrary, we obtain that 
    \[
    \sup_{x\in X} \dim \pi_x(Y\setminus \{x\}) \geq \min \{\dim X + \dim Y - k,k\}.
    \]
\end{proof}


\begin{remark}
    \rm{If one uses Corollary \ref{lundphamthu} instead of Theorem \ref{RnLundPhamThu} in this argument, the assumption that $X\cap Y = \emptyset$ may be dropped.} 
\end{remark}

\section{An Incidence Estimate}\label{sec3}

In the second proof of our main theorem, we first begin by generalizing Fu and Ren's planar incidence estimate to $\R^n$ for a Katz-Tao $(\de,\alpha)$-set of points and a Katz-Tao $(\de,\beta)$-set of tubes.

\begin{definition}
    For a set $A$, let $A^{(r)}$ denote the $r$-neighborhood of $A$. If $P$ is a set of $\delta$-balls and $\T$ is a set of $\delta$-tubes, let $P^r := \{ p^{(r)} : p \in P \}$ and $\T^r := \{ t^{(r)} : t \in \T \}$.
\end{definition}

\begin{prop}\label{prop:2.1}
Let $1 \le S \le D$, and fix $\eps > 0$. Let $P$ be a set of $1$-balls and $\T$ be a set of $1 \times D$-tubes in $[0, D]^n$. For $1 \le r \le S$, let $M_P (r) = \sup_{B_r} |P \cap B_r|$ and $M_\T (r) = \sup_{T_r} |\T \cap T_r|$. Then
\begin{equation}\label{eqn:2.1}
  \hspace{-.5cm}  I(P, \T) \lesim_\eps \left( \sum_{1 \le r \le S} r^{-3(n-1)/2} M_P (r)^{1/2} M_\T (r)^{1/2} \right) D^{(n-1)/2} |P|^{1/2} |\T|^{1/2} + S^{-1} I(P^{SD^\eps}, \T^{SD^\eps}).
\end{equation}
\end{prop}

Before presenting the proof, we construct a special bump-like function.

\begin{lemma}\label{lem:bump}
    There exists a function $\chi_0$ on $\R^n$ and a real number $c_n > 0$ with the following properties:
    \begin{itemize}
        \item $\chi_0 \ge 0$ on $\R^n$ and $\chi_0 (x) \ge c_d$ whenever $|x| \le 1$;

        \item $\hchi_0$ is supported in $B(0, 1)$.
    \end{itemize}
\end{lemma}

\begin{proof}
    Let $c = \frac{1}{100}$. Let $h$ be a non-negative $C^\infty$ bump function supported in $B(0, 2c)$ that takes the value $1$ on $B(0, c)$. Then $p(\xi) = h(\xi) * \overline{h(-\xi)}$ is a non-negative $C^\infty$ bump function supported in $B(0, 4c)$. Let $\chi_0 = \hat{p}$. We check the two properties:
    \begin{itemize}
        \item $\chi_0 (x) = \hath (x) \overline{\hath} (x) = |\hath (x)|^2 \ge 0$.

        \item If $|x| \le 1$, then $\Re(e^{-2\pi i x \cdot \xi}) > \frac{1}{2}$ for $|\xi| < 2c$. Hence, $\hath(x) \ge \frac{1}{2} c^n$.

        \item $\hchi_0 (\xi) = p(-\xi)$ is supported in $B(0, 2c) \subset B(0, 1)$.
    \end{itemize}
\end{proof}

For a $r$-tube $T$, the dual plate $T^*$ is the $r^{-1}$-plate orthogonal to $T$ through the origin. Also, let $d'(x, T) = d(\sigma(x), 0)$, where $\sigma$ is the affine map sending $T$ to the unit ball. Equivalently, for $x \notin 2T$, we have $d'(x, T) = \inf\{ s \ge 1: x \in sT \}$, where $sT$ is the dilate of $T$ around the center of $T$.

\begin{cor}\label{cor:char}
    For a $1 \times D$-tube $T$, there exists a non-negative function $\chi_T$ that is $1$ on $T$, has Fourier support on the dual plate $T^*$, and has rapid decay outside $t$: $\chi_t (x) \le \RapDec(1+d'(x, T))$.
\end{cor}

\begin{proof}[Proof of Proposition \ref{prop:2.1}]
    Cover $B_1$ by dyadic annuli $A_r = \{ \xi : |\xi| \in (\frac{r^{-1}}{2}, r^{-1}) \}$ for $1 \le r \le S$ and $A_S = \{ \xi : |\xi| \le S^{-1} \}$. Construct a smooth partition
    \begin{equation}\label{eqn:partition of unity}
        \chi_{B_1} = \sum_{S^{-1} < r < 1 \text{ dyadic}} \psi_r^2 + \psi_S^2.
    \end{equation}
    We may assume the $\psi_r$ are supported on $A_r$.

    Let $f = \sum_{p \in P} \chi_p$ and $g = \sum_{t \in \T} \chi_t$. Rewrite \eqref{eqn:partition of unity} as
    \begin{equation}\label{eqn:sum}
        \int_{\R^n} \chi_{B_1} \hatf\hatg \,d\xi = \int_{\R^n} \sum_{S^{-1} < r < 1 \text{ dyadic}} \psi_r^2 \hatf\hatg + \psi_S \hatf\hatg \,d\xi.
    \end{equation}
    By Plancherel, the left hand side of \eqref{eqn:sum} is $\int_{\R^n} (f * \cchi) g \,dx \sim I(P, \T)$.

    The rightmost term on the right hand side of \eqref{eqn:sum} is, by Plancherel,
    \begin{equation*}
        \int_{\R^n} (f * \cpsi_S)(g * \cpsi_S) \,dx = \sum_{p \in P} \sum_{t \in \T} \int_{\R^n} (\chi_p * \chi_S) \chi_t \,dx.
    \end{equation*}
    Note that $\chi_p * \chi_S (x) \lesim S^{-n} \RapDec(1+\frac{d(x,p)}{S})$ and $\chi_t (x) \lesim \RapDec(1+d(x,t))$. Hence, we have that $p \in P$ and $t \in \T$ contribute only if $d'(p, t) \le SD^{\eps}$, because otherwise the rapid decay destroys the contribution. When $d'(p, t) \le SD^\eps$, we get $S^{1-n}$. Finally, because $P, \T \subset [0, D]^n$, we have that $d'(p, t) \le SD^{\eps}$ implies $p^{SD^\eps} \cap T^{SD^\eps} \neq \emptyset$.

    Now we turn to the leftmost term on the right hand side of \eqref{eqn:sum}. For a given $r$, by Cauchy-Schwarz and Plancherel,
    \begin{equation}\label{eqn:c-s}
        \int_{\R^n} \psi_r^2 \hatf\hatg \, dx \le \left( \int_{\R^n} |\cpsi_r * f|^2 \, dx \right)^{1/2} \left( \int_{\R^n} |\cpsi_r * g|^2 \, dx \right)^{1/2}.
    \end{equation}
    For the first term in the product on the right hand side of \eqref{eqn:c-s}, we note that $\norm{\cpsi_r * f}_{L^\infty} \le r^{-n} M_P (r)$. Indeed, $|\cpsi_r * \chi_p (x)| \lesim r^{-n} \RapDec(1 + \frac{d(x, p)}{r})$. Since $M_P (s) \lesim \left( \frac{s}{r} \right)^n M_P (r)$, we have
    \begin{equation*}
        |\cpsi_r * f(x)| \lesim \sum_{s \ge r \text{ dyadic}} r^{-n} \RapDec\left(1 + \frac{d(x, p)}{s}\right) \left( \frac{s}{r} \right)^n M_P (r) \lesim r^{-n} M_P (r).
    \end{equation*}
    Also, by Young's inequality and the fact that $|\cpsi_r (x)| \le r^{-n} \RapDec(1 + \frac{|x|}{r})$ for any $x \in \R^n$ (so it has bounded $L^1$-norm),
    \begin{equation*}
        \norm{\cpsi_r * f}_{L^1} = \norm{f}_{L^1} \norm{\cpsi_r}_{L^1} \lesim |P|.
    \end{equation*}
    Thus,
    \begin{equation}\label{eqn:bound_f}
        \int_{\R^n} |\cpsi_r * f|^2 \, dx \le \norm{\cpsi_r * f}_{L^1} \norm{\cpsi_r * f}_{L^\infty} \le |P| r^{-n} M_P (r).
    \end{equation}

    For the second term in the product on the right hand side of \eqref{eqn:c-s}, we divide $\mathbb{S}^{n-1}$ into $\delta r$-caps $\theta$. For each $\theta$, let $\T_\theta$ be the tubes in $\T$ whose direction lies in the cap $\theta$. If a tube lies in multiple caps, choose one arbitrarily, so to ensure that $\T$ is the disjoint union of $\T_\theta$. If $g_\theta = \sum_{T \in \T_\theta} \chi_T$, then $g = \sum_\theta g_\theta$.

    \begin{lemma}\label{lem:angles}
        For each $1 \le r \le S$,
        \begin{equation*}
            \int_{A_r} |\sum_\theta \hatg_\theta|^2 \, d\xi \lesim \left( \frac{D}{r} \right)^{n-2} \cdot \sum_\theta \int_{A_r} |\hatg_\theta|^2
        \end{equation*}
    \end{lemma}
    
    \begin{proof}
        For any $|\xi| \sim r^{-1}$, it lies in $\lesim \left( \frac{D}{r} \right)^{n-2}$ many $r\delta$-plates through the origin. Now use Cauchy-Schwarz and the fact that $g_\theta$ is supported in some $r\delta$-plate.
    \end{proof}

    \begin{lemma}\label{lem:parallel}
        For each $r\delta$-cap $\theta$,
        \begin{equation*}
            \int |g_\theta * \cpsi_r|^2 \lesim r^{-(n-1)} M_\T (r) \cdot |\T_\theta| D.
        \end{equation*}
    \end{lemma}

    \begin{proof}
        We will show that $\norm{g_\theta * \cpsi_r}_{L^\infty} \lesim r^{-(n-1)} M_\T (r)$ and $\norm{g_\theta * \cpsi_r}_{L^1} \lesim |\T_\theta| D$.

        The second claim follows from triangle inequality and $$\norm{\chi_T * \cpsi_r}_{L^1} \le \norm{\chi_T}_{L^1} \norm{\cpsi_r}_{L^1} \lesim \norm{\chi_T}_{L^1} \lesim D.$$
        
        Let's turn to the first claim. Note that $\chi_T * \cpsi_r$ is essentially supported on a $r \times D$-tube with maximum amplitude $\norm{\chi_T * \cpsi_r}_\infty \lesim r^{-(n-1)}$ and rapid decay outside the support: for $y \in \R^n$, we have $\chi_T * \cpsi_r(x) \le r^{-(n-1)} \RapDec(1+\frac{d'(y, T)}{r})$. Fix $x \in \R^n$, and let $\bT$ be an $r$-tube parallel to $\theta$ centered at $x$. For any $s \ge r$, we can cover $S = \{ y \in \R^n : d'(y, \bT) \le s \} \cap B(0, 1)$ with $\lesim 1$ many $s$-tubes. Thus, there are $M_\T (s) \le M_\T (r) \cdot \left( \frac{s}{r} \right)^{2(n-1)}$ many tubes in $S$, and each such tube has  $|\chi_T * \cpsi_r (x)| \lesim r^{-(n-1)} \RapDec(1+\frac{s}{r})$.

        Now sum over dyadic $s \ge r$ to get
        \begin{equation*}
            |g_\theta * \cpsi_r (x)| \lesim \sum_{s \ge r \text{ dyadic}} r^{-(n-1)} \RapDec\left(1+\frac{s}{r}\right) M_\T (r) \cdot \left( \frac{s}{r} \right)^{2(n-1)} \lesim r^{-(n-1)} M_T (r).
        \end{equation*}
        This proves the first claim.
    \end{proof}

    Combining Lemmas \ref{lem:angles}, \ref{lem:parallel} and using Plancherel, we have
    \begin{equation}\label{eqn:bound_g}
        \int_{\R^n} |\cpsi_r * g|^2 \, dx \lesim D^{n-1} r^{-(2n-3)} M_\T (r) |\T|
    \end{equation}
    Now substitute \eqref{eqn:bound_f} and \eqref{eqn:bound_g} into \eqref{eqn:c-s} to finish.
\end{proof}

\begin{theorem}\label{incidence}
Let $\alpha\in [0,n]$ and $\beta\in [0,2(n-1)]$.
Let $\kappa = \min\left\{\frac{n-1}{\alpha + \beta -n+1}, \frac{1}{2}\right\}$. There exists $C > 0$ such that the following holds: for any Katz-Tao $(\delta, \alpha, K_\alpha)$-set of balls $P$ and Katz-Tao $(\delta, \beta, K_\beta)$-set of tubes $\T$, we have the following bound if $\alpha + \beta \neq 3(n-1)$:
\begin{equation*}
    I(P, \T) \le C \delta^{-(n-1)\kappa} (K_\alpha K_\beta)^\kappa |P|^{1-\kappa} |\T|^{1-\kappa}.
\end{equation*}
If $\alpha + \beta = 3(n-1)$, i.e. $\kappa = \frac{1}{2}$, we have
\begin{equation*}
    I(P, \T) \le C \log (\delta^{-1}) \delta^{-(n-1)/2} (K_\alpha K_\beta)^{1/2} |P|^{1/2} |\T|^{1/2}.
\end{equation*}
\end{theorem}

\begin{remark}
\rm{Note, that in either case, we have that 
\[
I(P,\T) \lesssim \de^{-\epsilon} \de^{-(n-1) \kappa} (K_\alpha K_\beta)^{\kappa} |P|^{1-\kappa} |\T|^{1-\kappa}.
\]
In this form, we arrive at a generalization of \cite[Theorem 1.5]{fu2022incidence}. However, Theorem \ref{incidence} gets rid of the $\delta^{-\eps}$ loss in many cases, and only imposes a $(\log \delta^{-1})$ loss in the critical case $\alpha + \beta = 3(n-1)$. By Construction 4 in \cite[Section 2]{fu2022incidence}, we see that at least in dimension $n = 2$, Theorem \ref{incidence} is sharp when $\alpha + \beta \neq 3$. When $\alpha + \beta = 3$, the $\log \delta^{-1}$ cannot be dropped at least in the case $\alpha = 1, \beta = 2$ due to the existence of Besicovitch sets, see e.g. \cite{keich1999lp}.}
\end{remark}

\begin{proof}[Proof of Theorem \ref{incidence}]
    Let $D = \delta^{-1}$ throughout this argument.

    Suppose $\alpha + \beta < 3(n-1)$. Apply Proposition \ref{prop:2.1} with $S = D$. Notice that $I(P^{S\delta^{-\eps}}, \T^{S\delta^{-\eps}}) \le |P| |\T|$. We also know $|P| \le K_\beta \delta^{-\beta}$ and $|\T| \le K_\alpha \delta^{-\alpha}$, so $$S^{1-n} I(P^{S\delta^{-\eps}}, \T^{S\delta^{-\eps}}) \le D^{-(n-1)} |P| |\T| \le D^{(n-1)/2} (K_\alpha K_\beta)^{1/2} |P|^{1/2} |\T|^{1/2}.$$
The summation in \eqref{eqn:2.1} becomes
    \begin{multline*}
        D^{(n-1)/2} |P|^{1/2} |\T|^{1/2} \sum_{1 \le r \le D \text{ dyadic}} r^{-3(n-1)/2} \cdot (K_\alpha r^\alpha)^{1/2} \cdot (K_\beta r^\beta)^{1/2} \\
        \lesim_{\alpha,\beta} D^{(n-1)/2} (K_\alpha K_\beta)^{1/2} |P|^{1/2} |\T|^{1/2}.
    \end{multline*}
    The fact $\alpha + \beta < 3(n-1)$ was used to generate a convergent geometric series.

    If $\alpha + \beta = 3(n-1)$, the same argument works, but the summation in \eqref{eqn:2.1} gives a $\log(\delta^{-1})$ factor.

    If $\alpha + \beta > 3(n-1)$, then let $$S = \max\left\{1, \left( \frac{|P| |\T|}{D^{n-1} K_\alpha K_\beta} \right)^{\kappa/(n-1)}\right\}.$$ Note that we have $S \le D^{\frac{(\alpha+\beta-(n-1))\kappa}{n-1}} = D$. As before, we have $$I(P^{S\delta^{-\eps}}, \T^{S\delta^{-\eps}}) \le |P| |\T|,$$ so 
    \[
    S^{-1} I(P^{S\delta^{-\eps}}, \T^{S\delta^{-\eps}}) \le S^{1-n} |P| |\T| \le D^{(n-1)\kappa} (K_\alpha K_\beta)^\kappa |P|^{1-\kappa} |\T|^{1-\kappa}.
    \]
Hence, the summation in \eqref{eqn:2.1} becomes
    \begin{align*}
        D^{(n-1)/2} |P|^{1/2} |\T|^{1/2} &\sum_{1 \le r \le S \text{ dyadic}} r^{-3(n-1)/2}(K_\alpha r^\alpha)^{1/2} \cdot (K_\beta r^\beta)^{1/2} \\
        &\lesim_{\alpha,\beta} D^{(n-1)/2} (K_\alpha K_\beta)^{1/2} |P|^{1/2} |\T|^{1/2} S^{(\alpha+\beta-3(n-1))/2} \\
        &\le D^{(n-1)\kappa} (K_\alpha K_\beta)^\kappa |P|^{1-\kappa} |\T|^{1-\kappa}.
    \end{align*}
\end{proof}

To obtain a result for $(\de,\alpha)$-sets of points and $(\de,\beta)$-sets of tubes, we apply the following ``decomposition'' lemma.

\begin{lemma}[\cite{orponen2022kaufman}, Lemma 3.5] \label{doubling}
    Let $(X,d)$ be a doubling metric space with constant $D\geq 1$. For every $\epsilon,t>0$, there exists $\de_0 = \de_0(\epsilon,D,t)>0$ such that the following holds for all $\de \in (0,\de_0]$. Let $P\subset B(x_0,1) \subset X$ be a $\de$-separated $(\de,\alpha,C)$-set of points in $X$. Then, we can partition $P$ such that 
    \[
    P = P_1\cup \dots \cup P_N
    \]
    where each $P_i$ is a Katz-Tao $(\de,\alpha,1)$-set and $N\leq C |P|\de^{\alpha-\epsilon}.$
\end{lemma}

We apply this result to Theorem \ref{incidence} and obtain the following.

\begin{theorem}\label{incidence2}
Let $\alpha\in [0,n]$ and $\beta\in [0,2(n-1)]$. Then, for every $\epsilon>0$, there exists a $\de_0 = \de_0(\epsilon)>0$ such that the following holds for all $\de \in (0,\de_0]$. If $P\subset B^n$ is a $(\de,\alpha,\de^{-\epsilon})$-set with $|P| \sim \delta^{-\alpha}$, and $\T$ is a $\de$-separated $(\de, \beta,\de^{-\epsilon})$-set with $|\T| \sim \delta^{-\alpha}$, then
    \[
    I(P,\T) \leq |P||\T| \de^{\kappa (\alpha+\beta-n+1) - 5\epsilon}
    \]
where $\kappa = \min\left\{\frac{n-1}{\alpha + \beta -n+1}, \frac{1}{2}\right \}$.
\end{theorem}

\begin{remark}
    \rm{In the proof of Theorem \ref{main_disc}, we let $\alpha = t \in (n-1,n]$ and $\beta = \gamma(\sigma,s) \in [0,2(n-1)]$. Note that these values of $\alpha$ and $\beta$ are in the domain of the above theorem.}
\end{remark}

\begin{proof}[Proof of Theorem \ref{incidence2}]
    By Lemma \ref{doubling} applied in both $\R^n$ and $\mathcal A(n,1)$, we may write 
    \[
    P = P_1\cup \dots \cup P_M \hspace{.4cm}\text{and} \hspace{.4cm}\T = \T_1\cup \dots \cup \T_N,
    \]
    where $M \leq |P|\de^{s-2\epsilon}$ and $N\leq \de^{t-2\epsilon}$, each $P_i$ is a Katz-Tao $(\de,s,1)$-set, and each $\T_j$ is a Katz-Tao $(\de,t,1)$-set. Hence, by \ref{incidence2}, we have 
    \begin{align*}
        I(P,\T) &\leq \sum_{i=1}^M \sum_{j=1}^N I(P_i,\T_j) \\
        &\leq \sum_{i=1}^M \sum_{j=1}^N \de^{-\epsilon} \de^{-(n-1)\kappa} |P||\T| \de^{\alpha\kappa} \de^{\beta\kappa} \\
        &= \sum_{i=1}^M \sum_{j=1}^N \de^{\kappa(\alpha + \beta -n+1) -\alpha - \beta - \epsilon} \\
        &\leq |P||\T| \de^{\alpha + \beta -4\epsilon} \de^{\kappa(\alpha + \beta -n+1) -\alpha - \beta - \epsilon} \\
        &= |P| |\T| \de^{\kappa(\alpha + \beta-n+1) -5\epsilon}.
     \end{align*}
\end{proof}

\section{A Dual Furstenberg Set Estimate}\label{sec4}

The purpose of this section is to prove a dual Furstenberg set estimate integral to the proof of Theorem \ref{main}, analogous to Lemma 3.13 in \cite{orponen2022kaufman}. This result will discuss $(\de,s,C)$-sets of $\de$-tubes $\mathbb{T}_x$, $x\in B^n$, such that $x\in T$ for all $T\in \mathbb{T}_x$. Note that one can check that the $(\de,s,C)$-set property of $\mathbb{T}_x$ is equivalent to the directions of the tubes forming a $(\de,s,C')$-set as a subset of $\mathbb{S}^{n-1}$.

We begin by proving a cardinality result for this dual Furstenberg set. This can be derived from \cite[Corollary 2.10]{ren2023discretized}, but we present the simple proof for completeness.

\begin{lemma}\label{furstenberg_cardinality}
    For every $\xi >0$, there exists $\de_0 = \de_0(\xi)>0$ and $\epsilon = \epsilon(\xi)>0$ such that the following holds for all $\de \in (0,\de_0]$. Let $\alpha\in [0,n-1]$ and $\beta \in (0,n].$ Assume that $P\subset B^n$ is a non-empty $(\de,\beta,\de^{-\epsilon})$-set. Further, assume that for every $p\in P$ there exists a $(\de, \alpha, \de^{-\epsilon})$-set of tubes $\T_p$ with the property that $p\in T$ for all $T\in \T_p$. Then, we have that
    \[
    \left|\bigcup_{p\in P} \T_p\right| \geq \de^{-\gamma(\alpha,\beta)+\xi}
    \]
    where $\gamma(\alpha,\beta) = \alpha + \min \{\alpha, \beta\},$ and $\left| \bigcup_{p\in P} \T_p \right|$ is the cardinality of the set of tubes $\bigcup_{p\in P} \T_p .$
\end{lemma}

\begin{proof}
    First, since $|\T_p| \in (\delta^{-\alpha+\eps}, \delta^{-1})$ for every $p \in P$, we may find a subset $P' \subset P$ with $|P'| \gesim (\log \delta^{-1})^{-1} |P|$ such that for each $p \in P'$, we have $|T_p| \in [M, 2M)$, for some $M \in (\delta^{-\alpha+\eps}, \delta^{-1})$. Next, we may prune each $|\T_p|$ for $p \in P'$ to have exactly $M$ elements, at the cost of making $\T_p$ a $(\delta, \alpha, 2\delta^{-\eps})$-set of tubes. Notice that $P'$ is still a $(\delta, \beta, \delta^{-2\eps})$-set. Now if we prove the lemma for $\eps$ and $\{ T_p \}_{p \in P'}$, then we also prove the lemma for $\frac{\eps}{2}$ and $\{ T_p \}_{p \in P}$. Hence, we may assume without loss of generality that $|\T_p| = M$ for all $p \in P$.

    Let $\T := \bigcup_{p\in P} \T_p$. Then, we find a lower bound for $|\T|$ by finding an upper and lower bound to the cardinality of the set 
    \[
    J(P,\T) := \{(p,p',T) \in P^2 \times \T \mid T \in \T_p \cap \T_{p'} \}.
    \]
    (This is similar to, but different from the set $\{(p,p',T) \mid p, p' \in T, T \in \T \}$, because there can be tubes $T \in \T$ passing through points $p$ in $P$, yet $T \not\in \T(p)$.)
Then we have that
\begin{align*}
|J(P, \T)| &= \sum_{T\in \T} \#\{(p,p') \in P^2 : T \in \T_p \cap \T_{p'} \} \\
&= \sum_T |\{ p \in P : T \in \T_p \}|^2 \\
&\geq |\T|^{-1} \left(\sum_{T} |\{ p \in P : T \in \T_p \}|\right)^2 \\
&= |\T|^{-1} \left(\sum_{p} |\T(p)|\right)^2 \\
&\geq |\T|^{-1}(|P| M)^2.
\end{align*}

We now find an upperbound to $J(P,\T')$ using a geometric argument.
\begin{align*}
    |J(P,\T')|&= \sum_p \sum_{p',T} \mathbf{1}(p\in T) \mathbf{1}(p'\in T)\\
    &\lesssim \sum_p \sum_{p'} \#(\T_p \cap \T_{p'}).
\intertext{Notice that $\#(\T_p \cap \T_{p'}) \lesim M\delta^{-\eps} \cdot (\frac{\delta}{d(p, p') + \delta})^\alpha$ since $\T_p$ is a $(\de,\alpha, \de^{-\epsilon})$-set and the set of tubes through both $p, p'$ lie in a common $\frac{\delta}{d(p, p') + \delta}$-tube. Since $P$ is a $(\de,\beta, \de^{-\epsilon})$-set, using a dyadic decomposition, there are $|P| \de^{-\epsilon} r^\beta$ many $\delta$-balls with $d(p, p') \sim r$. Hence, by summing over dyadic $r \in (\delta, 1)$, we get}
    &\lesssim \sum_{p} \sum_{\de \leq r \leq 1} |P|\delta^{-\eps} r^{\beta} \cdot M\delta^{-\eps} \cdot \left( \frac{\delta}{r} \right)^\alpha \\
    &\leq M|P|\delta^{\alpha-2\eps} \sum_{p} \sum_r r^{\beta - \alpha} \\
    &\lesim M|P|^2 \delta^{\alpha-2\eps} \max\{\de^{\beta - \alpha}, 1\}.
\end{align*}
Hence, combining the upper and lower bounds, we obtain 
\[
    |\T|^{-1} |P|^2 M^2 \leq |P|^2 M \delta^{\alpha-2\eps} \max\{\de^{\beta - \alpha}, 1\}.
\]
Therefore, using $M \ge \delta^{-\alpha+\eps}$, we get
\[
|\T| = \left|\bigcup_{x\in P} \T_x\right| \geq M \delta^{-\alpha+2\eps} \min\{\de^{\alpha - \beta}, 1\} \gesim \de^{-2\alpha + \max\{\alpha-\beta, 0\} + 3\eps}.
\]
Notice that $2\alpha - \max\{\alpha-\beta, 0\} = \alpha + \min\{\alpha,\beta\}$. This gives the desired result if we let $\eps := \frac{\xi}{4}$.
\end{proof}

Using this, and a standard Hausdorff content argument, we show that the union of tubes contains a $(\de,\gamma(\alpha,\beta), \de^{-\xi})$-set of tubes. First, we need a slight generalization of Proposition A.1 of \cite{Fassler_2014} but with the exact same proof. The statement is in fact an equality, the $\le$ direction being easier.
\begin{lemma}\label{lem:discrete frostman}
    In $[0, 1]^n$, let $P$ be a set of dyadic $\delta$-cubes, and let $f : \cD_{[\delta, 1]} \to [1, \infty)$ be a function, where $\cD_{[\delta, 1]}$ is the set of dyadic cubes with side length in $[\delta, 1]$. Then
    \begin{multline}\label{eqn:discrete frostman}
        \max \left\{ |P'| : P' \subset P, |P' \cap Q| \le f(Q) \, \forall Q \in \cD_{[\delta, 1]} \right\} \\
        \ge \min \left\{ \sum_{Q \in \cC} f(Q) \mid \cC \subset \cD_{[\delta, 1]}, \cC \text{ covers } P \right\}.
    \end{multline}
\end{lemma}

\begin{proof}
    Let $\delta = 2^{-k}$. For each $0 \le i \le k$, we construct a collection $\cC_i \subset \cD_{\le 2^{-i}}$ and a subset $P_i \subset \cC_i$ such that:
    \begin{enumerate}
        \item $P_0 \subset P_1 \subset \cdots \subset P_k = P$;

        \item for each $Q \in \cD_i$, we have $|P_i \cap Q| \le f(Q)$;
    
        \item for each $Q \in \cC_i$, we have $|P_i \cap Q| = f(Q)$.
    \end{enumerate}

    For $i = k$, we simply take $P_k = \cC_k = P$. Now given $\cC_{i+1}$ and $P_{i+1}$, we say that a square $Q \in \cD_r$ is rich for $P_{i+1}$ if $|P_{i+1} \cap Q| > 2^{(k-i) s}$; otherwise, $Q$ is poor.

    Let $\cC_i$ be the union of (elements of $\cC_{i+1}$ inside poor squares) and rich squares. Let $P_i$ be the union of (elements of $P_{i+1}$ inside poor squares) and, for each rich square $Q$, an arbitrary $2^{(k-i) s}$-element subset of $P_{i+1} \cap Q$. Then (1) is just true; (2) is true by definition of rich/poor square; and (3) is true by using (3) for $\cC_{i+1}$ and definition of rich/poor square.

    Now for any $r = 2^{-i}$ and $Q \in \cD_r$, we have by (1) and (2)
    \[
        |P_0 \cap Q| \le |P_i \cap Q| \le f(Q).
    \]

    Furthermore, by (3) we have $|P_0| = \sum_{Q \in \cC_0} |P_0 \cap Q| = \sum_{Q \in \cC_0} f(Q)$. Now $(P_0, \cC_0)$ is a witness to \eqref{eqn:discrete frostman}.
\end{proof}

\begin{corollary}\label{cor:discrete frostman A(n, 1)}
    Let $\T$ be a set of $\delta$-tubes, and let $s \ge 0$. Let $r(T)$ be the thickness of tube $T$, and let
    \begin{equation*}
        K = \min \left\{ \sum_{T \in \cC} r(T)^s \mid \cC \text{ is cover of } \T \text{ by dyadic } r'\text{-tubes, } r' \in [\delta, 1] \right\}.
    \end{equation*}
    Then $\T$ contains a $(\delta, s, O_n (1/K))$-set.
\end{corollary}

\begin{proof}
    Since $\mathcal{A}(n, 1)$ is locally homeomorphic to $[0, 1]^n$, we can apply Lemma \ref{lem:discrete frostman} with $f(T) = \left( \frac{r(T)}{\delta} \right)^s$. The maximizing set $\T'$ satisfies $|\T'| \ge K \cdot \delta^{-s}$. We check that $\T'$ is a $(\delta, s, O_n (1/K))$-set. Fix $T$; we can cover it with $O_n (1)$ many dyadic tubes of radius $\sim r(T)$, so without loss of generality $T$ is dyadic. Then by the condition on $\T'$, we have
    \[
        |P' \cap T| \lesssim_n \left( \frac{r(T)}{\delta} \right)^s \le \frac{1}{K} |P'| \cdot r(T)^s.
    \]
\end{proof}

\begin{lemma}\label{furstenberg}
    For every $\xi >0$, there exists $\de_0 = \de_0(\xi)>0$ and $\epsilon = \epsilon(\xi)>0$ such that the following holds for all $\de \in (0,\de_0]$. Let $\alpha\in [0,n-1]$ and $\beta \in (0,n].$ Assume that $P\subset B^n$ is a non-empty $(\de,\beta,\de^{-\epsilon})$-set. Further, assume that for every $p\in P$ there exists a $(\de, \alpha, \de^{-\epsilon})$-set of tubes $\T_p$ with the property that $x\in T$ for all $T\in \T_p$. Then, the union 
    \[
    \T := \bigcup_{p\in P} \T_p
    \]
    contains a $(\de,\gamma(\alpha,\beta), O_N(\delta^{-\xi}))$-set of tubes $\T'$, where $\gamma(\alpha,\beta) = \alpha + \min \{\alpha, \beta\}$. 
\end{lemma}

\begin{proof}
    The proof of this statement essentially follows that of \cite[Lemma 3.3]{hera2020improved}, though we work through the details here. 
    
    By Corollary \ref{cor:discrete frostman A(n, 1)}, it suffices to check that the RHS of Corollary \ref{cor:discrete frostman A(n, 1)} is large. Fix a cover $\cC$ of $\T$ by dyadic tubes with thickness between $\de$ and $1$ (i.e. a cover satisfying the conditions of Corollary \ref{cor:discrete frostman A(n, 1)}). By assumption, for every $p\in P$, there exists a set $\T_p$ of tubes that is a $(\de,\alpha, \de^{-\epsilon})$-set. Given that $\cC$ is a cover of $\T$ and $\T_p \subset \T$, we know that $\cC$ is a cover of $\T_p$.

    For each $p\in P$, the tubes in $\cC$ covering $\T_p$ may be of different thicknesses, but by dyadic pigeonholing we may pick a ``popular'' sized dyadic thickness for tubes covering $\T_p.$ Rigorously, for each $p$ there exists an $r(p)$ such that $\cC_{r(p)} \subset \cC$ is a set of tubes of dyadic thickness $r(p)$ covering $\T_p$ with $|\T_p \cap \cC_{r(p)}| \ge \frac{1}{\log \delta^{-1}} |\T_p|$. Then $\T_p \cap \cC_{r(p)}$ is a $(\delta, \alpha, \delta^{-2\eps})$-set.
    
    We now dyadically pigeonhole the $r(p)$. In particular, there exists an $r \in [\de,1]$ such that 
    \[
    P' = |\{p \in P : r(p) = r\}| \geq \frac{1}{\log \de^{-1}}|P|.
    \]
    Then $P'$ is a $(\delta, \beta, \delta^{-2\eps})$-set. Let $\cC_r\subset \cC$ be the subset of tubes of thickness $r$. 
    
    Therefore, we may apply Lemma \ref{furstenberg_cardinality} to $\cC_r$, and we conclude that there exists an $\epsilon$ such that $|\cC_r| \ge r^{-\gamma(\alpha, \beta)+\xi}$, so $$\sum_{T \in \cC_r} r(T)^{\gamma(\alpha, \beta)} \ge |\cC_r| \cdot r^{\gamma(\alpha, \beta)} \ge \delta^{\xi}.$$ Hence, since $\cC$ is an arbitrary cover satisfying the RHS of Corollary \ref{cor:discrete frostman A(n, 1)}, it follows that $K\geq \de^\xi$. Hence, $\T$ contains a $(\de, \gamma(\alpha,\beta), O_n(\de^{-\xi}))$-set $\T$.
\end{proof}

\begin{remark}
    \rm{In the proof of Theorem \ref{main_disc}, we let $\alpha = \sigma \in [0,n-1]$, and $\beta = s \in (0,n]$. Note that these values of $\alpha$ and $\beta$ are in the domain of the above theorem.}
\end{remark}

We remark that, by Lemma \ref{furstenberg} and a verbatim proof of \cite[Lemma 3.3]{hera2020improved}, the following corollary regarding dual Furstenberg sets of lines holds. 

\begin{corollary}
    Let $\alpha \in [0,n-1]$ and $\beta \in (0,n]$. Suppose $E$ is a $\beta$-dimensional subset of $\R^n$ and $\mathcal{L}$ is a collection of lines with the property that for every point $x \in E$ there exists a $\alpha$-dimensional collection of lines $\mathcal{L}_x \subset \mathcal{L}$ such that every line in $\mathcal{L}_x$ passes through $x$. Then $\dim \mathcal{L} \geq \gamma(\alpha, \beta)$ where $\gamma(\alpha, \beta) = \alpha + \min \{\alpha, \beta \}$.   
\end{corollary}

\section{Proof of Theorem \ref{main}: The Case $k=n-1$}\label{sec5}

We now combine the results of Sections \ref{sec3} and \ref{sec4} to prove the following discretized version of Theorem \ref{main} when $k=n-1$.

\begin{theorem}\label{main_disc}
    For all $t\in (n-1,n]$, $\sigma \in [0,n-1)$, and $\zeta >0$, there exists $\epsilon = \epsilon(\sigma, t,\zeta) >0$  and $\de_0 \in 2^{-\N}$ small enough such that the following holds for all $\de \in (0,\de_0]$.

    Let $s\in [0,n]$. Let $P_K\subset B^n$ be a $\de$-separated $(\de, t,\de^{-\epsilon})$-set, and let $P_E \subset B^n$ be a $\de$-separated $(\de, s, \de^{-\epsilon})$-set. Assume that for every $x\in P_E$, there exists a $(\de, \sigma, \de^{-\epsilon})$-set of tubes $\T_x$ with the properties $x\in T$ for all $T\in \T_x$, and 
    \[
    |T\cap P_K| \geq \de^{\sigma + \epsilon}|P_K|
    \]
    for all $T\in \T_x$. Then, $\sigma \geq s+t - n+1 - \zeta$.
\end{theorem}

Before proving this discretized result, we first recall some important preliminaries from the paper of Orponen, Shmerkin, and Wang.

\subsection{Preliminaries}

We write $\mathcal P(X)$ for the family of Borel probability measures on a metric space $(X,d)$. We use the following notation: if $X,Y\subset \R^n$ and $G\subset X\times Y$, we write 
\[
G|_x := \{y\in Y : (x,y) \in G\} \hspace{.25cm} \text{and} \hspace{.25cm} G|^y := \{x\in X : (x,y) \in G\}
\]
for $x\in X$ and $y\in Y$. If $G\subset \R^n \times \R^n$ is a Borel set, then so too are $G|_x$ and $G|^y$ as $\Bor(\R^n \times \R^n) = \Bor(\R^n) \times \Bor(\R^n).$

\begin{definition}[Thin tubes] Let $K,t\geq 0$ and $c\in (0,1]$. Let $\mu,\nu\in \mathcal P(\R^n)$ with $\spt(\mu) =: X$ and $\spt(\nu) =: Y$. We say that $(\mu,\nu)$ has $(t,K,c)$-\textbf{thin tubes} if there exists a Borel set $G\subset X\times Y$ with $(\mu \times \nu) (G) \geq c$ with the following property: If $x\in X$, then 
\[
\nu(T\cap G|_x) \leq K \cdot r^t \text{~~for all $r>0$ and all $r$-tubes $T$ containing $x$.}
\]
We also say that $(\mu,\nu)$ has $t$-thin tubes if $(\mu,\nu)$ has $(t,K,c)$-thin tubes for some $K,c>0$.
\end{definition}

We also recall Remark 2.3 from Orponen, Shmerkin, and Wang due to its importance in our proofs.
\begin{remark}[\cite{orponen2022kaufman}, Remark 2.3] \label{thintubesimplication}
    \rm{Assume that $\mu,\nu\in \mathcal P(\R^n)$ has $t$-thin tubes for some $t>0$. Then, there exists $x\in \spt (\mu)$ such that $\mathcal H^t(\pi_x(Y\setminus \{x\}))>0$. In other words, this implies 
    \[
    \sup_{x\in X} \dim \pi_x(Y\setminus\{x\}) \geq t.
    \]}
\end{remark}

\subsection{Proof of Theorem \ref{main_disc}}

We prove Theorem \ref{main_disc} analogous to Orponen, Shmerkin, and Wang with a slight change in numerology.

\begin{proof}[Proof of Theorem \ref{main_disc}]
    Fix $s\in [0,n]$, $t\in (n-1,n]$, $\sigma \in [0,n-1)$, and $\zeta>0$. Let $P_K$ and $P_E$ be as in the statement of the theorem-- namely, let $P_K \subset B^n$ be a $\de$-separated $(\de,t,\de^{-\epsilon})$-set and let $P_E\subset B^n$ be a $\de$-separated $(\de, s, \de^{-\epsilon})$-set. Furthermore, suppose that for every $x\in P_E$, there exists a $(\de, \sigma, \de^{-\epsilon})$-set of tubes $\T_x$ with the properties that $x\in T$ for all $T\in \T_x$, and 
    \begin{equation}\label{furst2}
    |T\cap P_K|\geq \de^{\sigma + \epsilon}|P_K|
    \end{equation}
    for all $T\in \T_x$. Given $\T_x$ is a $(\de,\sigma,\de^{-\epsilon})$-set of tubes and $P_K$ is a $(\de, s, \de^{-\epsilon})$-set of points, by Lemma \ref{furstenberg}, $\T' = \bigcup_{x\in P}\T_x$ contains a $(\de, \gamma(\sigma,s),\de^{-\xi})$-set of tubes, which we denote by $\mathbb{T},$ where $\gamma(\sigma,s) = \sigma + \min\{s,\sigma\}$. 
    
    We now state our conditions on $\epsilon$:
    \begin{equation}\label{epsilon_reqs}
        10 \xi + 2\epsilon \leq \zeta, \text{~~~and~~~} \sigma < n-1 - 5\xi - \epsilon, \text{~~~and~~~} \frac{t-n+1}{2} - 5\xi - \epsilon >0.
    \end{equation}
    Such an $\epsilon$ exists as $\sigma < n-1$ and $t>n-1$.

    By \eqref{furst2}, we have 
    \[
    |P_K| |\T| \de^{\sigma + \epsilon} \leq \sum_{T\in \T} |T\cap P_K| = I(P_K,\T).
    \]
    We now compare this lowerbound against the upperbounds from Theorem \ref{incidence}. Since $P_K$ is a $(\de,t)$-set and $\T$ is a $(\de, \gamma(\sigma,s))$-set, we consider the quantity:
    \[
    \overline{\kappa} = \kappa(t, \gamma(\sigma,s)) = \min \left\{\frac{1}{2}, \frac{n-1}{t-n+1+ \gamma(\sigma,s)}\right\}.
    \]
We break up into three cases.
\begin{enumerate}
    \item[i)] Suppose that $\overline \kappa = \frac{n-1}{t-n+1 + \gamma(\sigma, s)}$. Then, regardless of the value of $\gamma(\sigma,s)$, we have 
    \[
    |P_K||\T|\de^{\sigma + \epsilon} \leq I(P_K,\T) \leq |P_K||\T| \de^{\overline{\kappa} (t-n+1 + \gamma(\sigma,s)) -5\xi} = |P_K||\T|\de^{n-1 - 5\xi}
    \]
    Hence, $\sigma \geq n-1-5\xi$.
    By the assumption \eqref{epsilon_reqs}, this is a contradiction.

    \item[ii)] Suppose that $\overline{\kappa} = 1/2$. Then, we break into cases.
    \begin{enumerate}
        \item Assume $s\leq \sigma$ so $\gamma(\sigma,s) = \sigma + s$. Then, we have 
        \[
        |P_K||\T|\de^{\sigma + \epsilon} \leq I(P_K,\T) \leq |P_K||\T|\de^{\overline{\kappa} (t-n+1 + \gamma(\sigma,s) )-5\xi} = |P_K||\T|\de^{(t-n+1 + \sigma+s)/2 - 5\xi}.
        \]
        Hence, $$\sigma \geq s+t -n+1 - 10\xi - 2 \epsilon \geq s+ t-n+1 - \zeta$$ where here we use the assumption that $\zeta \geq 10 \xi - 2\epsilon$. This gives the desired result.
        \item Assume now that $s\geq \sigma$, so $\gamma(\sigma,s) = 2\sigma$. Then, we have 
        \[
|P_K||\T|\de^{\sigma + \epsilon} \leq I(P_K,\T) \leq |P_K||\T|\de^{\overline{\kappa} (t-n+1 + 2\sigma) -5\xi}.
        \]
        Hence, $\sigma \geq (t-n+1)/2 -5\xi - \epsilon$, which contradicts \eqref{epsilon_reqs}.
    \end{enumerate}
\end{enumerate}
This completes the proof of Theorem \ref{main_disc}.    
 \end{proof}

\subsection{A Discretized Version of Theorem \ref{main}}

We now prove Theorem \ref{thintubesversion}.

\begin{theorem}\label{thintubesversion}
    Let $s\in [0,n]$, $t\in (n-1,n]$, $0\leq \sigma < \min \{s + t - n+1, n-1\}$, $C>0$, and $\epsilon \in (0,1]$. Then, there exists $K = K(C, \epsilon, s, \sigma, t) >0$ such that the following holds. Assume $\mu,\nu \in \mathcal P(B^n)$ satisfy 
    \[
    \mu(B(x,r)) \leq Cr^s \text{~~and~~} \nu(B(x,r)) \leq Cr^t.
    \]
    Then, $(\mu,\nu)$ has $(\sigma, K,1-\epsilon)$-thin tubes.

    In particular, whenever $s\in [0,n]$, $t\in (n-1,n]$, $I_s(\mu)<\infty$ and $I_t(\nu)<\infty$, then $(\mu,\nu)$ has $\sigma$-thin tubes for all $0 \leq \sigma < \min \{s + t-n+1,n-1\}$.
\end{theorem}

We establish the necessary $\epsilon$-improvement of the result, Lemma \ref{epsilon_improve} below, which iterated multiple times implies Theorem \ref{thintubesversion}. 

\begin{lemma}\label{epsilon_improve}
Let $s\in [0,n]$, $t\in (n-1,n]$, and $0\leq \sigma < \min \{s+ t-n+1,n-1\}$. Let $\epsilon \in (0,\frac{1}{10})$ and $C,K>0.$ Let $\mu,\nu\in \mathcal P(B^n)$ such that $\mu(B(x,r)) \leq Cr^s$ and $\nu(B(y,r)) \leq Cr^t$ for all $x,y\in \R^n$ and $r>0$. If $(\mu,\nu)$ has $(\sigma , K, 1-\epsilon)$-thin tubes, then there exists $\eta = \eta(s,\sigma,t)>0$ and $K' >0$ such that $(\mu,\nu)$ has $(\sigma + \eta, K', 1-4\epsilon)$-thin tubes. Furthermore, $\eta(s,\sigma,t)$ is bounded away from zero on any compact subset of 
\[
\Omega := \{(s,\sigma,t) \in [0,n]\times [0,n-1]\times (n-1,n] : \sigma < \min \{s + t-n+1, n-1\}\}.
\]
\end{lemma}

\begin{proof}[Proof of Theorem \ref{thintubesversion} assuming Lemma \ref{epsilon_improve}]
The proof is nearly identical to that of Orponen, Shmerkin, and Wang's Theorem 3.20, with the main difference that our set $\Omega$ is properly adjusted for this higher dimensional setting. In particular, we fix $0\leq \sigma < \min \{s + t - n+1, n-1\}$ in this proof. Still, we include the argument for completeness.

We automatically have that $(\mu,\nu)$ has $(t-1,K_0,1)$-thin tubes for some $K_0 \sim C$ by the Frostman condition on $\nu$, as we can take a $r$-tube $T$ and break it into $r^{-1}$ many delta balls, obtaining $\nu(T) \leq C \cdot \frac{r^t}{r} = Cr^{t-1}$. Therefore, $(\mu,\nu)$ has $(t-1, K_0,1-\overline{\epsilon})$-thin tubes for every $\epsilon \in (0,\frac{1}{10})$. Now, if $\min \{s + t-n+1,n-1\}\leq t-1$, by monotonicity of the definition of thin tubes, we are done. Otherwise, we apply Lemma \ref{epsilon_improve}. Fix $0\leq \sigma < \min \{s + t-n+1,n-1\}$, and let 
\[
\eta:= \inf \{\eta(s,\sigma', t) : t-1 \leq \sigma'\leq \sigma\},
\]
where $\eta(s,\sigma,t)$ comes from Lemma \ref{epsilon_improve}.
Note that $\eta >0$ as $\eta$ is bounded away from zero on compact subsets of $\Omega$. Let 
\[
\overline{\epsilon} = \epsilon \cdot 4^{-1/\eta}/100
\]
where $\epsilon$ is given in the statement. 

Applying Lemma \ref{epsilon_improve} once, we get that $(\mu,\nu)$ has $(t-1 + \eta,K_1, 1-4\overline{\epsilon})$-thin tubes for some $K_1 >0$. If $t-1 + \eta >\sigma$, we are done, as by monotonicity of thin tubes this implies $(\mu,\nu)$ has $\sigma$-thin tubes. Otherwise, as apply the application again. We repeat this process at most $N$ steps such that $N \leq 1/\eta$. So, we have that $(\mu,\nu)$ has 
\[
(t - 1 + 2^N \eta, K_N, 1-4^N \overline{\epsilon})\text{-thin tubes} \implies (\sigma, K_N, 1-4^N \overline{\epsilon})\text{-thin tubes}.
\]
Hence $(\mu, \nu)$ has $(\sigma, K_N, 1- 4^N \overline{\epsilon})$-thin tubes, which implies $(\mu, \nu)$ has $(\sigma, K_N, \epsilon)$-thin tubes. This completes the proof.    
\end{proof}

It still remains to prove Lemma \ref{epsilon_improve}.

\begin{proof}[Proof of Lemma \ref{epsilon_improve}]
    We only sketch the argument, as this result follows nearly directly from Orponen, Shmerkin, and Wang's proof of Lemma 3.21 in \cite{orponen2022kaufman}. First, we suppose for the sake of contradiction that $(\mu,\nu)$ does not have $(\sigma + \eta, K', 1-4\epsilon)$-thin tubes for $K'\geq 1$ and $\eta >0$ to be determined. Since $\sigma < s + t - n+1$, we may choose $\zeta = \zeta(s,\sigma,t) >0$ such that 
    \[
    \sigma < s + t-n+1 - \zeta.
    \]
    We then conclude the proof by applying Theorem \ref{main_disc} to reach a contradiction. To apply Theorem \ref{main_disc}, we need reduce our collection of points and tubes to ones that satisfy the hypotheses of Theorem \ref{main_disc}.

    This reduction is most of the proof of Lemma 3.21 in \cite{orponen2022kaufman}, and the method of reduction follows identically in higher dimensions. The conclusion of this reduction and Theorem \ref{main_disc} is that 
    \[
    \sigma \geq s + t -n+1 - \zeta,
    \]
    which is a contradiction.
\end{proof}

We now prove that Theorem \ref{main_disc} implies Theorem \ref{main}.

\begin{proof}
    Given $X \neq \emptyset$ and $\dim Y>n-1$, choose Frostman measures $\mu,\nu$ such that $\spt(\mu) = X$ and $\spt(\nu) = Y$. Then, by Theorem \ref{main_disc}, we have that $(\mu,\nu)$ has $\sigma$-thin tubes for all $0\leq \sigma < \min \{s + t- n+1,n-1\}$. Therefore, by Remark \ref{thintubesimplication}, this implies 
    \[
    \sup_{x\in X} \dim \pi_x(Y\setminus \{x\}) \geq \sigma.
    \]
    As $\sigma < \min \{s + t- n+1,n-1\}$ was arbitrary, this shows that
    \[
    \sup_{x\in X} \dim \pi_x(Y\ \{x\}) \geq \min \{\dim X + \dim Y  - n + 1,n-1\}.
    \]
\end{proof}

\section{Proof of Theorem \ref{main}}\label{sec6}

We now prove Theorem \ref{main} for all values of $k$ using a Marstrand projection theorem and one of the radial projection results of Orponen-Shmerkin-Wang \cite[Theorem 1.9.i]{orponen2022kaufman}. We restate the theorem here.
\begin{theorem} \label{main_final}
    Let $X,Y\subset \R^n$ be Borel sets with $X\neq \emptyset$, $\dim Y\in (k, k+1]$ for some $k \in \{0,1,\dots, n-1\}$. Then, 
    \[
    \sup_{x\in X} \dim \pi_x(Y\setminus \{x\}) \geq \min \{\dim X + \dim Y-k, k\}.
    \]
\end{theorem}

\begin{proof}
    Note that if $k = n-1$, this result is precisely Theorem \ref{main}, so we suppose $k<n-1$.
    
    We break this in proof into two cases. Firstly, suppose that $\dim X > k+1$. In this case, then 
    \[
    \min \{\dim X + \dim Y-k, k\} = k
    \]
as $\dim Y\in (k,k+1]$. In this case, we apply Theorem 1.9.i from \cite{orponen2022kaufman} (which we may do as $\dim X >k+1$), obtaining that
\begin{align*}
    \sup_{x\in X} \dim \pi_x(Y\setminus \{x\}) &= \dim Y \\
    &\geq k  \\
    &= \min \{\dim X + \dim Y-k, k\}.
\end{align*}
This concludes the result.

Otherwise, suppose $\dim X\leq  k+1$. Suppose for the sake of contradiction that there exists subsets $X,Y\subset \R^n$ with $X\neq \emptyset$, $\dim X \leq k+1$, and $\dim Y \in (k,k+1]$ such that 
\[
\sup_{x\in X} \dim \pi_x(Y\setminus \{x\}) < \min \{\dim X + \dim Y - k, k\}.
\]
Then by \cite[Corollary 9.4]{mattila_1995}, for $\gamma_{n,k}$ almost every $V\in G(n,k+1)$, 
\[
\dim P_V(X) = \dim X \hspace{.25cm} \text{and} \hspace{.25cm} \dim P_V(Y) = \dim Y,
\]
where $P_V$ is the orthogonal projection onto the subspace $V.$

On the other hand, 
\begin{align*}
    \sup_{x\in X} \dim  \pi_{P_V(x)} (P_V(Y)) &\leq \sup_{x\in X} \dim \pi_x(Y\setminus \{x\}) \\
    &< \min \{\dim X + \dim Y - k, k\} \\
    &= \min \{\dim P_V(X) + \dim P_V(Y) - k, k\}.
\end{align*}   
This contradicts the $\R^{k+1}$ version of Theorem \ref{main} (proven in Section \ref{sec5}) in $V\simeq \R^{k+1}$ applied to $P_V(X)$ and $P_V(Y)$. This concludes the proof.
\end{proof}



\bibliographystyle{abbrv}
\bibliography{ref}

\end{document}